\documentclass[11pt,a4paper]{article}

\usepackage{amsmath}
\usepackage{amsthm}
\usepackage{amssymb}
\usepackage{amscd}
\usepackage[all]{xy}

\title{Kawamata-Viehweg Vanishing on 
Rational Surfaces in Positive Characteristic}
\author{Qihong Xie}
\date{}
\theoremstyle{plain}
\newtheorem{prop}{Proposition}[section]
\newtheorem{lem}[prop]{Lemma}
\newtheorem{thm}[prop]{Theorem}
\newtheorem{cor}[prop]{Corollary}

\theoremstyle{definition}
\newtheorem{defn}[prop]{Definition}
\newtheorem*{ack}{Acknowledgments}
\newtheorem*{nota}{Notation}

\theoremstyle{remark}
\newtheorem{rem}[prop]{Remark}

\newcommand{\Q}{\mathbb Q}

\newcommand{\Z}{\mathbb Z}
\newcommand{\N}{\mathbb N}
\newcommand{\F}{\mathbb F}
\newcommand{\A}{\mathbb A}
\newcommand{\PP}{\mathbb P}
\newcommand{\OO}{\mathcal O}
\newcommand{\II}{\mathcal I}

\newcommand{\LL}{\mathcal L}

\newcommand{\TT}{\mathcal T}

\newcommand{\Supp}{\mathop{\rm Supp}\nolimits}
\newcommand{\Exc}{\mathop{\rm Exc}\nolimits}
\newcommand{\ch}{\mathop{\rm char}\nolimits}
\newcommand{\Hom}{\mathop{\rm Hom}\nolimits}
\newcommand{\Spec}{\mathop{\rm Spec}\nolimits}
\newcommand{\Proj}{\mathop{\bf Proj}\nolimits}

\newcommand{\ext}{\mathop{\rm Ext}\nolimits}
\newcommand{\ra}{\rightarrow}

\newcommand{\wt}{\widetilde}

\begin{document}
\maketitle

\begin{abstract}
We prove that the Kawamata-Viehweg vanishing theorem holds on 
rational surfaces in positive characteristic by means of the 
lifting property to $W_2(k)$ of certain log pairs on smooth 
rational surfaces. As a corollary, the Kawamata-Viehweg vanishing 
theorem holds on log del Pezzo surfaces in positive characteristic.
\end{abstract}

\setcounter{section}{0}
\section{Introduction}\label{S1}

There are many generalizations of the celebrated Kodaira vanishing theorem. 
One of the most important generalizations is the Kawamata-Viehweg vanishing 
theorem. As is well known, it is inevitable to run the higher dimensional 
minimal model program in the categories of varieties with suitable 
singularities, hence we have to consider $\Q$-divisors instead of integral 
divisors. It turns out that the Kawamata-Viehweg vanishing theorem is 
indispensable and plays a crucial role in birational geometry of higher 
dimensional algebraic varieties.

The Kawamata-Viehweg vanishing theorem is of several forms. The one dealing 
with ample $\Q$-divisors follows directly from the Kodaira vanishing theorem 
via the Kummer covering trick \cite{ka82,vi}.

\begin{thm}[Kawamata-Viehweg vanishing]\label{1.1}
Let $X$ be a smooth projective variety over an algebraically 
closed field $k$ with $\ch(k)=0$, and $H$ an ample $\Q$-divisor 
on $X$ such that the fractional part $\langle H\rangle$ has simple 
normal crossing support. Then $H^i(X,K_X+\ulcorner H\urcorner)=0$ 
holds for any $i>0$.
\end{thm}

The most general form is stated for log pairs which have only Kawamata log 
terminal singularities \cite[Theorem 1-2-5]{kmm}.

\begin{thm}[Kawamata-Viehweg vanishing]\label{1.2}
Let $X$ be a normal projective variety over an algebraically closed 
field $k$ with $\ch(k)=0$, $B=\sum b_iB_i$ an effective $\Q$-divisor 
on $X$, and $D$ a $\Q$-Cartier Weil divisor on $X$. 
Assume that $(X,B)$ is Kawamata log terminal (KLT for short), 
and $D-(K_X+B)$ is ample. Then $H^i(X,D)=0$ holds for any $i>0$.
\end{thm}

The original proof of the Kodaira vanishing theorem was analytic, and its 
purely algebraic proof was first given by Deligne and Illusie \cite{di}. 
For a smooth proper variety $X$ over a perfect field $k$ of characteristic 
$p>0$, they have defined the notion of a lifting of $X$ to $W_2(k)$, 
the ring of Witt vectors of length two of $k$, and have proved that 
if $X$ admits a lifting to $W_2(k)$ and $\dim X<p$, then the de Rham 
complex is decomposable in derived category, consequently the Hodge 
to de Rham spectral sequence degenerates in $E_1$, and the 
Kodaira-Akizuki-Nakano vanishing theorem holds on $X$. 
The characteristic zero case can be deduced from the 
characteristic $p$ case by standard arguments \cite[\S 6]{il}. 
Furthermore, they have also claimed the corresponding results 
\cite[\S 4.2]{di} in the logarithmic case for a log pair $(X,D)$, 
where $X$ is a smooth proper variety and $D\subset X$ is a simple 
normal crossing divisor over $k$. The explicit statements and proofs 
of those results were given by Esnault and Viehweg \cite[\S 8-\S11]{ev}. 
In particular, if $(X,D)$ admits a lifting to $W_2(k)$, then the 
logarithmic Kodaira-Akizuki-Nakano vanishing theorem holds on $X$.

Later, Hara \cite{hara} and Matsuki and Olsson \cite{mo} independently 
proved the Kawamata-Viehweg vanishing theorem in positive characteristic 
under the lifting condition  to $W_2(k)$ of certain log pairs. The method 
of Hara is to use a quasi-isomorphism between the logarithmic de Rham 
complex and its variant by adding certain modulo $p$ fractional parts, 
while Matsuki and Olsson replaced the Kummer covering trick with 
the stack technique, which behaves well in arbitrary characteristic, 
and interpreted the Kawamata-Viehweg vanishing on varieties as 
the Kodaira vanishing on stacks.

\begin{thm}[Kawamata-Viehweg vanishing in char.\ $p>0$]\label{1.3}
Let $k$ be a perfect field of characteristic $p>0$, $X$ a smooth 
projective variety over $k$ of dimension $d$, $H$ an ample 
$\Q$-divisor on $X$, and $D$ a simple normal crossing divisor 
containing $\Supp(\langle H\rangle)$. 
Assume that $(X,D)$ admits a lifting to $W_2(k)$. Then 
\[ H^i(X,\Omega_X^j(\log D)(-\ulcorner H\urcorner))=0 \,\,\,\,
\hbox{holds for any} \,\,\,\, i+j<\inf(d,p). \]
In particular, $H^i(X,K_X+\ulcorner H\urcorner)=0$ holds 
for any $i>d-\inf(d,p)$.
\end{thm}

The lifting condition to $W_2(k)$, together with the reduction modulo 
$p$ technique, is usually used to prove some statements in characteristic 
zero. However, the lifting condition is indeed a very strong condition, 
since it is not satisfied even for some log pairs with simple structure 
(see Corollary \ref{1.10}).

In what follows, we always work over {\it an algebraically closed 
field $k$ of characteristic $p>0$} unless otherwise stated. 
The following main theorem, i.e.\ the Kawamata-Viehweg vanishing 
theorem on rational surfaces, will be proved in this paper.

\begin{thm}\label{1.4}
Let $X$ be a normal projective rational surface, 
$D$ a $\Q$-Cartier Weil divisor on $X$, and 
$B$ an effective $\Q$-divisor such that $(X,B)$ is KLT 
and $D-(K_X+B)$ is ample. Then $H^1(X,D)=0$ holds.
\end{thm}

Thanks to Theorem \ref{1.3}, we have only to verify that the lifting 
condition to $W_2(k)$ holds for certain log pairs on smooth rational 
surfaces. The main idea of the proof is to reduce the problem to the 
Hirzebruch surface case.

\begin{defn}\label{1.5}
A pair $(X,B)$ is called a log del Pezzo surface, if $X$ is a normal 
projective surface, and $B$ is an effective $\Q$-divisor on $X$ 
such that $(X,B)$ is KLT and $-(K_X+B)$ is ample.

A normal projective surface $X$ is called a log del Pezzo surface 
(resp.\ weak log del Pezzo surface), if $(X,0)$ is KLT and $-K_X$ 
is ample (resp.\ nef and big).
\end{defn}

There are some corollaries of the main theorem.

\begin{cor}\label{1.6}
Let $(X,B)$ be a log del Pezzo surface, $D$ a $\Q$-Cartier Weil 
divisor on $X$ such that $D-(K_X+B)$ is ample. Then $H^1(X,D)=0$ holds.
\end{cor}

\begin{cor}\label{1.7}
Let $X$ be a (weak) log del Pezzo surface. Then $H^1(X,\OO_X)=0$ holds.
\end{cor}

\begin{rem}\label{1.8}
A Fano variety, by definition, is a projective variety $X$ with the 
anticanonical divisor $-K_X$ ample. Fano surface is conventionally 
called del Pezzo surface. As is well known, Fano variety has 
appeared as a kind of outcome of running the minimal model program, 
so the study of Fano varieties is of certain interest in birational 
geometry of algebraic varieties. Let us recall some known vanishing 
or non-vanishing results concerning Fano varieties in positive 
characteristic, which show that Corollary \ref{1.6} is just a result 
as expected.

(1) Tango \cite{ta} has proved that the Kodaira vanishing 
theorem does hold on smooth projective ruled surfaces, hence on 
{\it smooth} del Pezzo surface.

(2) Reid \cite{re} has found {\it nonnormal} del Pezzo surfaces 
$X$ with $H^1(X,\OO_X)\neq 0$.

(3) Schr\"{o}er \cite{sc} proved that over any {\it nonperfect} 
field $k$ of characteristic $p=2$, there is a normal del Pezzo 
surface $X$ with $H^1(X,\OO_X)\neq 0$.

(4) Shepherd-Barron \cite{sb} established that $H^1(X,\OO_X)=
H^2(X,\OO_X)=0$ holds for smooth Fano {\it threefolds}.

(5) Lauritzen and Rao \cite{lr} has constructed counterexamples 
to the Kodaira vanishing theorem on some smooth Fano varieties of 
{\it dimension at least 6}.
\end{rem}

Theorem \ref{1.4} also implies the following corollary, which is a weak 
version of the logarithmic Koll\'{a}r vanishing theorem 
\cite[Theorem 10.19]{ko} and the logarithmic semipositivity theorem 
\cite[Theorem 1.2 and Corollary 1.3]{ka00} on rational surfaces.

\begin{cor}\label{1.9}
Let $X$ be a normal projective rational surface, $f:X\ra \PP^1$ a 
surjective proper morphism, and $B$ an effective $\Q$-divisor on $X$ 
such that $(X,B)$ is KLT. Let $D$ be a $\Q$-Cartier Weil divisor on $X$ 
such that $D-(K_X+B)$ is ample. Then 

(1) $H^1(\PP^1,R^if_*\OO_X(D))=0$ holds for any $i\geq 0$, and 

(2) $f_*\OO_X(D-f^*K_{\PP^1})$ is an ample vector bundle on $\PP^1$.
\end{cor}

Unfortunately, Theorem \ref{1.4} and Corollary \ref{1.9} fail for 
certain ruled surfaces (see \cite[Examples 3.7,3.9,3.10]{xie06}). 
As a consequence, it follows that the lifting condition to $W_2(k)$ is 
not satisfied even for some log pairs on geometrically ruled surfaces 
(see \cite[Definition 2.6]{xie07} for the definition of Tango curve).

\begin{cor}\label{1.10}
If $C$ is a Tango curve, then there are a $\PP^1$-bundle $f:X\ra C$ 
and a smooth curve $C'\subset X$ such that $(X,C')$ cannot be 
lifted to $W_2(k)$.
\end{cor}

In \S \ref{S2}, we will prove some results concerning the 
lifting property of certain log pairs on smooth rational surfaces. 
\S \ref{S3} is devoted to the proofs of the main theorem and the 
corollaries. Finally, we will give some remarks on the main results 
in \S \ref{S4}. For the necessary notions and results in birational 
geometry, e.g.\ Kawamata log terminal singularity, we refer 
the reader to \cite{kmm} and \cite{km}.

\begin{nota}
We use $\equiv$ to denote numerical equivalence, and $[B]=\sum [b_i] B_i$ 
(resp.\ $\ulcorner B\urcorner=\sum \ulcorner b_i\urcorner B_i$, 
$\langle B\rangle=\sum \langle b_i\rangle B_i$, $\{B\}=\sum\{b_i\}B_i$ ) 
to denote the round-down (resp.\ round-up, fractional part, upper fractional 
part) of a $\Q$-divisor $B=\sum b_iB_i$, where for a real number $b$, 
$[b]:=\max\{ n\in\Z \,|\,n\leq b \}$, $\ulcorner b\urcorner:=-[-b]$, 
$\langle b\rangle:=b-[b]$ and $\{ b\}:=\ulcorner b\urcorner-b$.
\end{nota}

\begin{ack}
I would like to express my gratitude to Professor Yujiro Kawamata 
for valuable advice and warm encouragement. 
I would also like to thank Professors Noboru Nakayama and Nobuo Hara 
for useful comments. 
I am indebted to the referee for giving a simple proof of Lemma \ref{2.3}. 
This work was partially supported by the 21st Century COE Program 
and the Global COE Program in the University of Tokyo.
\end{ack}

\section{Lifting property on smooth rational surfaces}\label{S2}

Let us first recall some definitions from \cite[Definition 8.11]{ev}.

\begin{defn}\label{2.1}
Let $W_2(k)$ be the ring of Witt vectors of length two of $k$. 
Then $W_2(k)$ is flat over $\Z/p^2\Z$, and $W_2(k)\otimes_{\Z/p^2\Z}\F_p=k$. 
For the explicit construction and further properties of $W_2(k)$, 
we refer the reader to \cite[II.6]{se}. The following definition 
generalizes the definition \cite[1.6]{di} of liftings of $k$-schemes 
to $W_2(k)$.

Let $X$ be a noetherian scheme over $k$, and $D=\sum D_i$ a reduced Cartier 
divisor on $X$. A lifting of $(X,D)$ to $W_2(k)$ consists of a scheme 
$\wt{X}$ and closed subschemes $\wt{D_i}\subset\wt{X}$, all defined and 
flat over $W_2(k)$ such that $X=\wt{X}\times_{\Spec W_2(k)}\Spec k$ and 
$D_i=\wt{D_i}\times_{\Spec W_2(k)}\Spec k$. We write 
$\wt{D}=\sum \wt{D_i}$ and say that $(\wt{X},\wt{D})$ is a lifting 
of $(X,D)$ to $W_2(k)$, if no confusion is likely.
\end{defn}

In the above definition, assume further that $X$ is smooth 
over $k$ and $D=\sum D_i$ is simple normal crossing. 
If $(\wt{X},\wt{D})$ is a lifting of $(X,D)$ to $W_2(k)$, 
then $\wt{X}$ is smooth over $W_2(k)$ and $\wt{D}=\sum \wt{D_i}$ 
is simple normal crossing over $W_2(k)$, i.e.\ $\wt{X}$ 
is covered by affine open subsets $\{U\}$, such that each 
$U$ is \'{e}tale over $\A^n_{W_2(k)}$ via coordinates 
$\{ x_1,\cdots,x_n \}$ and $\wt{D}|_U$ is defined by 
the equation $x_1\cdots x_\nu=0$ with $1\leq\nu\leq n$ 
(see \cite[Lemmas 8.13,8.14]{ev}).

If $\wt{X}$ is a lifting of $X$ to $W_2(k)$, then there is an exact 
sequence of $\OO_{\wt{X}}$-modules
\[ 0\ra \OO_X\stackrel{p}{\ra} \OO_{\wt{X}}\stackrel{r}{\ra} \OO_X\ra 0, \]
where $p(x):=px$ and $r(\wt{x}):=\wt{x}\mod p$ 
for $x\in\OO_X,\wt{x}\in\OO_{\wt{X}}$ (see \cite[Lemma 8.13]{ev}).

For instance, $\A^n_k$, $\PP^n_k$ and 
$H_m=\PP(\OO_{\PP^1_k}\oplus\OO_{\PP^1_k}(-m))$ have liftings to $W_2(k)$.

\begin{defn}\label{2.2}
Let $X$ be a smooth scheme over $k$, $D=\sum D_i$ a reduced divisor on 
$X$, and $Z$ a closed subscheme of $X$ smooth over $k$ of codimension 
$s\geq 2$. A mixed lifting of $(X,D+Z)$ to $W_2(k)$ consists of a smooth 
scheme $\wt{X}$ over $W_2(k)$, closed subschemes $\wt{D_i}\subset\wt{X}$ 
flat over $W_2(k)$, and a closed subscheme $\wt{Z}\subset\wt{X}$ smooth 
over $W_2(k)$ such that $X=\wt{X}\times_{\Spec W_2(k)}\Spec k$, 
$D_i=\wt{D_i}\times_{\Spec W_2(k)}\Spec k$ and 
$Z=\wt{Z}\times_{\Spec W_2(k)}\Spec k$. 
We write $\wt{D}=\sum \wt{D_i}$ and say that $(\wt{X},\wt{D}+\wt{Z})$ 
is a mixed lifting of $(X,D+Z)$ to $W_2(k)$, if no confusion is likely.
\end{defn}

In the above definition, either $D=\emptyset$ or $Z=\emptyset$ is allowed. 
Obviously, if $Z=\emptyset$ then a mixed lifting $(\wt{X},\wt{D})$ of 
$(X,D)$ is indeed a lifting of $(X,D)$ to $W_2(k)$.

For instance, if $X=\A_k^n$ or $\PP_k^n$ or $H_m$, and $P\in X$ is 
a closed point (or an infinitesimal closed point), then $(X,P)$ has 
a mixed lifting to $W_2(k)$.

We need the following elementary lemmas.

\begin{lem}\label{2.3}
Let $X=H_m=\PP(\OO_{\PP^1_k}\oplus\OO_{\PP^1_k}(-m))$ be a 
Hirzebruch surface with $m\geq 0$. Then for any reduced divisor 
$D$ on $X$, $(X,D)$ has a mixed lifting to $W_2(k)$.
\end{lem}

\begin{proof}
Since $X$ has a natural lifting $\wt{X}=\wt{H_m}=\PP(\OO_{\PP^1_{W_2(k)}}
\oplus\OO_{\PP^1_{W_2(k)}}(-m))$, we have only to lift the irreducible 
components of $D$ to $W_2(k)$ one by one. Thus we may assume that $D$ is 
irreducible. Let $f:X\ra\PP^1$ be the natural projection. Take a section 
$E$ of $f$ with $\OO_X(E)\cong\OO_X(1)$ and $E^2=-m\leq 0$.

If $D.E<0$ then we have $D=E$ and $E^2<0$. In this case, 
$D$ has a lifting $\wt{D}$, which is the unique curve on $\wt{X}$ 
with negative self-intersection. 

From now on, assume $D.E\geq 0$. The following exact sequence of 
abelian sheaves: 
\[ 0\ra \OO_X\stackrel{q}{\ra} \OO_{\wt{X}}^*\stackrel{r}{\ra} 
\OO_X^*\ra 1, \]
where $q(x):=1+px$ for $x\in\OO_X$, gives rise to the exact sequence 
$H^1(\wt{X},\OO_{\wt{X}}^*)\ra H^1(X,\OO_X^*)\ra H^2(X,\OO_X)=0$. 
Therefore, the invertible sheaf $\LL:=\OO_X(D)$ on $X$ extends to an 
invertible sheaf $\wt{\LL}$ on $\wt{X}$. Let $s\in H^0(X,\LL)$ be a 
section corresponding to the divisor $D$. Then lifting $D$ is nothing 
but to extend the section $s$ to a section $\wt{s}\in H^0(\wt{X},\wt{\LL})$. 
The long exact sequence associated to $0\ra \LL\ra \wt{\LL}\ra \LL\ra 0$ 
shows that it suffices to prove $H^1(X,\LL)=0$.

Write $D\sim aE+bF$, where $F$ is the fiber of $f$, $a\geq 0$ and $b\geq am$. 
We use induction on $a$ to prove that $H^1(X,\OO_X(aE+bF))=0$ holds for any 
$a\geq 0$ and $b\geq am$. When $a=0$, we have $H^1(X,\OO_X(bF))\cong 
H^1(\PP^1,\OO_{\PP^1}(b))=0$. Assume $a>0$. The exact sequence 
$H^1(X,\OO_X((a-1)E+bF))\ra H^1(X,\OO_X(aE+bF))\ra H^1(E,\OO_E(b-am))$ 
and the induction hypothesis conclude the argument.
\end{proof}

\begin{lem}\label{2.4}
Let $X$ be a smooth scheme over $k$, $D$ a reduced divisor on $X$, 
and $Z\subset X$ a closed subscheme smooth over $k$ of codimension 
$s\geq 2$. Let $\pi:X'\ra X$ be the blow-up of $X$ along $Z$ with the 
exceptional divisor $E$, $D'=\pi_*^{-1}D$ the strict transform of $D$. 
Assume that $(X,D+Z)$ admits a mixed lifting to $W_2(k)$. 
Then $(X',D'+E)$ admits a mixed lifting to $W_2(k)$.
\end{lem}

\begin{proof}
Let $(\wt{X},\wt{D}+\wt{Z})$ be a mixed lifting of $(X,D+Z)$ to $W_2(k)$. 
Then $\wt{Z}\subset\wt{X}$ is a closed subscheme smooth over $W_2(k)$ of 
codimension $s\geq 2$. Let $\wt{I}$ be the ideal sheaf of $\wt{Z}$ in 
$\wt{X}$, $\wt{\pi}:\wt{X}'\ra\wt{X}$ the blow-up of $\wt{X}$ along $\wt{Z}$ 
with the exceptional divisor $\wt{E}$, and $\wt{D}'=\wt{\pi}_*^{-1}\wt{D}$. 
By \cite[Corollary II.7.15]{ha}, we have the following commutative 
diagram:
\[
\xymatrix{
X'' \ar[d]_{\pi'} \ar@{^{(}->}[r] & \wt{X}' \ar[d]^{\wt{\pi}} \\
X \ar@{^{(}->}[r] & \wt{X} 
}
\]
where $\pi':X''\ra X$ is the blow-up of $X$ with respect to the ideal 
sheaf $\wt{I}\otimes_{W_2(k)}k=I$, the ideal sheaf of $Z$ in $X$. 
Hence $X''=X'$ and $\pi'=\pi$. Since $\wt{X}$ is smooth over $W_2(k)$, 
so is $\wt{X}'$. Note that $\wt{X}'\times_{\Spec W_2(k)}
\Spec k=\Proj(\oplus_i \wt{I}^i)\times_{\Spec W_2(k)}\Spec k
=\Proj(\oplus_i \wt{I}^i\otimes_{W_2(k)}k)=\Proj(\oplus_i I^i)=X'$,
so $\wt{X}'$ is a lifting of $X'$ to $W_2(k)$. It is easy to see that 
$\wt{D}'\times_{\Spec W_2(k)}\Spec k=D'$ and 
$\wt{E}\times_{\Spec W_2(k)}\Spec k=E$, hence 
$(X',D'+E)$ has a mixed lifting $(\wt{X}',\wt{D}'+\wt{E})$ to $W_2(k)$.
\end{proof}

\begin{defn}\label{2.5}
Let $X$ be a smooth projective surface, and $D$ a reduced divisor on $X$. 
$D$ is said to be suitable if there exists a birational morphism 
$f:X\ra X_{min}$ such that 

(1) $f$ is the composition of some $(-1)$-curve contractions,

(2) $X_{min}$ is a relatively minimal model, and 

(3) $D$ contains the exceptional locus $\Exc(f)$ of $f$.
\end{defn}

\begin{prop}\label{2.6}
Let $X$ be a smooth projective rational surface over $k$, 
$D=\sum_{j=1}^r D_j$ a suitable simple normal crossing divisor on $X$. 
Then $(X,D)$ admits a lifting to $W_2(k)$.
\end{prop}

\begin{proof}
If $\rho(X)=1$, then $X\cong\PP^2_k$ and the conclusion is obvious. 
From now on, we may assume $\rho(X)\geq 2$. By assumption, 
there is a sequence of $(-1)$-curve contractions:
\[ X=X_n\stackrel{(-1)}{\longrightarrow} X_{n-1}\stackrel{(-1)}
{\longrightarrow}\cdots\stackrel{(-1)}{\longrightarrow} X_1
\stackrel{(-1)}{\longrightarrow} X_0, \]
where $X_0$ is a Hirzebruch surface, say $H_m$ with $m\geq 0$. 

Let $E_i\subset X_i$ be the corresponding $(-1)$-curves whose 
images are the smooth closed points $P_{i-1}\in X_{i-1}$ ($1\leq i\leq n$), 
$\pi_i: X\ra X_i$ the induced morphisms ($0\leq i\leq n$), and 
$E_i'=\pi_{i*}^{-1}E_i$ the strict transforms on $X$ ($1\leq i\leq n$). 
By assumption, $\sum_{i=1}^n E_i'$ is contained in $D=\sum_{j=1}^r D_j$. 
Let $D^i=\pi_{i*}D$, $0\leq i\leq n-1$. Then in general the irreducible 
components of $D^0$ are neither smooth nor intersect transversally.

First of all, we assume $P_i\in D^i$ for all $0\leq i\leq n-1$. 
Then $\pi_0:D\subset X\ra D^0\subset X_0$ is a procedure consisting 
of a sequence of one point blow-ups such that the support of the total 
transform of $D^0$ is equal to the support of $D$, which is simple 
normal crossing. 

By Lemma \ref{2.3}, $(X_0,D^0)$ has a mixed lifting $(\wt{X_0},\wt{D^0})$ 
to $W_2(k)$. Let $\eta:D^0\hookrightarrow\wt{D^0}$ be the induced closed 
immersion, and let $\wt{P_0}=\eta(P_0)\in\wt{D^0}$. 
If $P_0\in X_0$ is locally defined by equations $x=x_0,y=y_0$, 
then $\wt{P_0}$ is locally defined by equations $x=\wt{x_0},y=\wt{y_0}$ 
with $r(\wt{x_0})=x_0$, $r(\wt{y_0})=y_0$, where $x_0,y_0\in k$, $\wt{x_0},
\wt{y_0}\in W_2(k)$. Therefore $(X_0,D^0+P_0)$ has a mixed lifting 
$(\wt{X_0},\wt{D^0}+\wt{P_0})$ to $W_2(k)$. 
By Lemma \ref{2.4}, $(X_1,D^1)$ has a mixed lifting $(\wt{X_1},\wt{D^1})$ 
to $W_2(k)$. We can repeat the same argument as above and use the induction 
on $n$ to prove that $(X,D)$ has a mixed lifting ($\wt{X},\wt{D}$) to 
$W_2(k)$, which is indeed a lifting of $(X,D)$ to $W_2(k)$.

In general, if $P_i\not\in D^i$ for some $i$, then $P_i$ is isolated from 
$D^0$ (we denote the image of $P_i$ in $X_0$ by the same symbol), and we 
can further prove that $(X_0,P_i)$ has a mixed lifting to $W_2(k)$, 
hence so does $(X_i,D^i+P_i)$. The rest is the same as above.
\end{proof}

\section{Proof of the main theorem}\label{S3}

The following vanishing result \cite[Corollary 2.2.5]{kk} is useful, 
which holds in arbitrary characteristic.

\begin{lem}\label{3.1}
Let $h: Y\ra X$ be a proper birational morphism between normal surfaces 
with $Y$ smooth and with exceptional locus $E=\cup_{i=1}^s E_i$. 
Let $L$ be an integral divisor on $Y$, $0\leq b_1,\cdots,b_s<1$ 
rational numbers, and $N$ an $h$-nef $\Q$-divisor on $Y$. Assume 
$ L\equiv K_Y+\sum_{i=1}^s b_iE_i+N$. Then $R^1h_*\OO_Y(L)=0$ holds.
\end{lem}

We can use Lemma \ref{3.1} to show that the KLT surface singularity 
is rational in positive characteristic, while the general statement 
that the KLT singularity is rational in characteristic zero has been 
proved in \cite[Theorem 5.22]{km}.

\begin{lem}\label{3.4}
Let $X$ be a normal proper surface, and $B$ an effective $\Q$-divisor 
on $X$ such that $(X,B)$ is KLT. Then $X$ has only rational singularities.
\end{lem}

\begin{proof}
Let $h: Y\ra X$ be the minimal resolution of $X$. Write 
$K_Y \equiv h^*K_X+\sum_{i=1}^s a_iE_i$ with $-1<a_i\leq 0$ 
for all $i$, and $h_*^{-1}B \equiv h^*B+\sum_{i=1}^s c_iE_i$ 
with $c_i\leq 0$ for all $i$. Hence we have $K_Y+h_*^{-1}B 
\equiv h^*(K_X+B)+\sum_{i=1}^s b_iE_i$ with $b_i=a_i+c_i\leq 0$. 
Since $(X,B)$ is KLT, $b_i>-1$ holds for all $i$. Since 
$0 \equiv K_Y+\sum_{i=1}^s (-b_i)E_i+h_*^{-1}B-h^*(K_X+B)$, 
by Lemma \ref{3.1}, we have $R^1h_*\OO_Y=0$. It is easy to 
see that $R^1h_*\omega_Y=0$ holds.
\end{proof}

\begin{proof}[Proof of Theorem \ref{1.4}]
Take a log resolution $h: Y\ra X$ such that 

(1) $Y$ is a smooth projective rational surface over $k$, 
and we can write $K_Y+h_*^{-1}B \equiv h^*(K_X+B)+\sum_ia_iE_i$, 
where $E_i$ are the exceptional curves of $h$ and $a_i>-1$ 
for all $i$.

(2) $G=\Supp(h_*^{-1}B)\cup\Exc(h)\cup($some self-intersection 
negative curves on $Y$) is suitable and simple normal crossing.

Let $D_Y=\ulcorner h^*D+\sum_ia_iE_i\urcorner$. Since 
$\ulcorner \sum_ia_iE_i\urcorner\geq 0$ is supported by $\Exc(h)$, 
we have $h_*\OO_Y(D_Y)=\OO_X(D)$ by the projection formula. 
Since $\{h^*D+\sum_ia_iE_i\}$ is supported by $\Exc(h)$, 
we can take $0<\delta_i\ll 1$ such that

(1) $\bigl{[}h_*^{-1}B+\{h^*D+\sum_ia_iE_i\}+
\sum_i\delta_iE_i\bigr{]}=0$.

(2) $D_Y-(K_Y+h_*^{-1}B+\{h^*D+\sum_ia_iE_i\}+
\sum_i\delta_iE_i) \equiv h^*(D-(K_X+B))-\sum_i\delta_iE_i$ is ample.

Let $B_Y=h_*^{-1}B+\{h^*D+\sum_ia_iE_i\}+\sum_i\delta_iE_i$. 
Then $H_Y=D_Y-(K_Y+B_Y)$ is ample, $\Supp(\langle H_Y\rangle)=\Supp(B_Y)$ 
is simple normal crossing, and $K_Y+\ulcorner H_Y\urcorner=D_Y$. Note that
\[ D_Y \equiv K_Y+\{h^*D+\sum_ia_iE_i\}+h^*(D-(K_X+B))
+h_*^{-1}B. \]
By Lemma \ref{3.1}, we have $R^1h_*\OO_Y(D_Y)=0$, hence 
$H^1(Y,D_Y)=H^1(X,h_*\OO_Y(D_Y))=H^1(X,D)$. 

Since $G$ is a suitable simple normal crossing divisor on the smooth 
rational surface $Y$, by Proposition \ref{2.6}, $(Y,G)$ admits a lifting 
to $W_2(k)$. Since $G$ contains $\Supp(\langle H_Y\rangle)$, we have 
$H^1(X,D)=H^1(Y,D_Y)=H^1(Y,K_Y+\ulcorner H_Y\urcorner)=0$ 
by Theorem \ref{1.3}.
\end{proof}

\begin{proof}[Proof of Corollary \ref{1.6}]
It follows from the cone theorem \cite[2.1.1 and 2.1.4]{kk} that the 
Kleiman-Mori cone $\overline{NE}(X)$ is generated by rational curves. 
By Lemma \ref{3.4}, $X$ has only rational singularities, therefore 
$X$ is rational. The rest is due to Theorem \ref{1.4}.
\end{proof}

\begin{cor}\label{3.2}
Let $X$ be a weak log del Pezzo surface, and $D$ a $\Q$-Cartier Weil 
divisor on $X$ such that $D-K_X$ is nef and big. Then $H^1(X,D)=0$ holds.
\end{cor}

\begin{proof}
Take an effective $\Q$-divisor $B_1$ such that $(X,B_1)$ is KLT 
and $-(K_X+B_1)$ is ample. Then $X$ is rational by the same argument 
as above. Take another effective $\Q$-divisor $B_2$ such that 
$(X,B_2)$ is KLT and $D-(K_X+B_2)$ is ample. The rest is due to 
Theorem \ref{1.4}.
\end{proof}

\begin{proof}[Proof of Corollary \ref{1.7}]
It follows from Theorem \ref{1.4} or Corollary \ref{3.2}.
\end{proof}

\begin{cor}\label{3.3}
Let $X$ be a smooth projective rational surface, $f:X\ra\PP^1$ a surjective 
projective morphism, and $H$ an $f$-ample $\Q$-divisor on $X$ such that 
the fractional part $\langle H\rangle$ has simple normal crossing support. 
Then $R^1f_*\OO_X(K_X+\ulcorner H\urcorner)=0$ holds.
\end{cor}

\begin{proof}
By assumption, there exists an $m\in\N$ such that $mH$ is integral 
and the natural morphism $f^*f_*\OO_X(mH)\ra\OO_X(mH)$ is surjective, 
which induces a closed immersion $\varphi:X\ra\PP(f_*\OO_X(mH))$ with 
$mH=\varphi^*\OO(1)$. Therefore $H$ is ample on $X$.

Let $P$ be a general point in $\PP^1$, $F=f^{-1}(P)$ the general fibre 
of $f$, and $m$ a positive integer. Consider the Leray spectral sequence 
$E_2^{ij}=H^i(\PP^1,R^jf_*\OO_X(K_X+\ulcorner H\urcorner+mF))
\Rightarrow H^{i+j}(X,\OO_X(K_X+\ulcorner H\urcorner+mF))$. 
By Serre vanishing, $E_2^{ij}=0$ holds for any $i>0$ and any $m\gg 0$. 
Therefore we have $H^0(\PP^1,R^1f_*\OO_X(K_X+\ulcorner H\urcorner+mF))
=H^1(X,\OO_X(K_X+\ulcorner H\urcorner+mF))=0$ by Theorem \ref{1.4}. 
Note that $R^1f_*\OO_X(K_X+\ulcorner H\urcorner+mF)=
R^1f_*\OO_X(K_X+\ulcorner H\urcorner)\otimes\OO_{\PP^1}(m)$ 
is generated by global sections for $m\gg 0$, so we have 
$R^1f_*\OO_X(K_X+\ulcorner H\urcorner)=0$.
\end{proof}

\begin{proof}[Proof of Corollary \ref{1.9}]
(1) We proceed a similar argument to the proof of Theorem \ref{1.4} 
to obtain a log resolution $h:Y\ra X$ from a smooth projective 
rational surface $Y$, a divisor $D_Y$ and a $\Q$-divisor 
$B_Y$ on $Y$, such that $f\circ h$ is projective, $H_Y=D_Y-(K_Y+B_Y)$ 
is ample, $[B_Y]=0$ and $\Supp(B_Y)$ is simple normal crossing. 
Furthermore, we have $R^1h_*\OO_Y(D_Y)=0$ and 
$h_*\OO_Y(D_Y)=\OO_X(D)$.

Let $g=f\circ h:Y\ra\PP^1$ be the induced morphism. 
It follows from Corollary \ref{3.3} that $R^1f_*\OO_X(D)=
R^1g_*\OO_Y(D_Y)=R^1g_*\OO_Y(K_Y+\ulcorner H_Y\urcorner)=0$. 
By the Leray spectral sequence and Theorem \ref{1.4}, 
we have $H^1(\PP^1,f_*\OO_X(D))=H^1(X,D)=0$.

(2) Since $\OO_X(D)$ is torsion free, so is $f_*\OO_X(D)$. 
Hence $f_*\OO_X(D)$ is a locally free sheaf on $\PP^1$. 
By Grothendieck's theorem (cf.\ \cite{oss}), 
$f_*\OO_X(D)$ is a direct sum of invertible sheaves on $\PP^1$: 
$f_*\OO_X(D)=\OO_{\PP^1}(d_1)\oplus\cdots\oplus\OO_{\PP^1}(d_n).$
Note that
\begin{eqnarray}
H^1(X,D) & = & H^1(\PP^1,f_*\OO_X(D))
=\bigoplus_iH^1(\PP^1,\OO_{\PP^1}(d_i)), \nonumber \\
f_*\OO_X(D-f^*K_{\PP^1}) & = & f_*\OO_X(D)\otimes\omega_{\PP^1}^{-1}
=\bigoplus_i\OO_{\PP^1}(d_i+2), \nonumber
\end{eqnarray}
so the vanishing of $H^1(X,D)$ implies the ampleness of 
$f_*\OO_X(D-f^*K_{\PP^1})$.
\end{proof}

\section{Some remarks on the main results}\label{S4}

First of all, we recall the following criterion for the liftability 
of log pairs.

\begin{lem}\label{4.1}
Let $X$ be a smooth variety, and $D$ a simple normal crossing divisor on $X$. 
Then there is an obstruction 
$o(X,D)\in\ext_{\OO_X}^2(\Omega_X^1(\log D),\OO_X)=H^2(X,\TT_X(-\log D))$ 
to the liftability of $(X,D)$ to $W_2(k)$, i.e.\ $o(X,D)=0$ if and only if 
$(X,D)$ is liftable to $W_2(k)$.
\end{lem}

\begin{proof}
The case when $D=\emptyset$ was verified directly in 
\cite[Proposition 2.12]{il}. For the general case, 
\cite[Proposition 8.22]{ev} just showed that the isomorphisms of liftings 
of $(X,D)$ over an open subset form a ``torseur'' under the group 
$\Hom_{\OO_X}(\Omega_X^1(\log D),\OO_X)$, hence by a similar argument to 
that of \cite[Proposition 2.12]{il}, we get the required obstruction 
$o(X,D)\in\ext_{\OO_X}^2(\Omega_X^1(\log D),\OO_X)=H^2(X,\TT_X(-\log D))$. 
An alternative proof follows from a deep result in \cite[4.2.3]{di}: 
the ``gerbe'' $\mathrm{rel}(X,D,W_2(k))$ of liftings of $(X,D)$ is 
canonically equivalent to the ``gerbe'' 
$\mathrm{sc}(\tau_{\leq 1}F_*\Omega_X^\bullet(\log D))$ of splittings of 
the complex $\tau_{\leq 1}F_*\Omega_X^\bullet(\log D)$. Hence we have 
$o(X,D)=\mathrm{cl\,\,sc}(\tau_{\leq 1}F_*\Omega_X^\bullet(\log D))
\in\ext_{\OO_X}^2(\Omega_X^1(\log D),\OO_X)$ by \cite[3.2]{di}.
\end{proof}

Abelian varieties and complete intersections in $\PP_k^n$ are liftable 
to $W_2(k)$, which are nontrivial results of Grothendieck and Deligne 
(see \cite[7.11]{il}). On the other hand, from Lemma \ref{4.1}, 
it follows easily that any smooth projective curve (or any log pair on it) 
is liftable to $W_2(k)$. Furthermore, we have the following consequence:

\begin{cor}\label{4.2}
Let $X$ be a smooth projective surface with $\kappa(X)<0$. 
Then $X$ is liftable to $W_2(k)$.
\end{cor}

\begin{proof}
If $X\cong\PP^2_k$ then the conclusion is obvious. So we may assume that 
there is a fibration $f:X\ra C$ over a smooth projective curve $C$ with 
a general fiber $F\cong\PP^1$. By Lemma \ref{4.1} and Serre duality, it 
suffices to show $H^0(X,\Omega_X^1\otimes\omega_X)=0$. Suppose to the 
contrary that $H^0(X,\Omega_X^1\otimes\omega_X)\neq 0$ holds, then we can 
take a section $0\neq s\in H^0(X,\Omega_X^1\otimes\omega_X)$ such that 
$0\neq s|_F\in H^0(F,\Omega_X^1\otimes\omega_X|_F)$. Let $\II=\OO_X(-F)$ be 
the ideal sheaf of $F$ in $X$. Then we have $\II/\II^2=\OO_X(-F)|_F\cong\OO_F$. 
By adjunction formula, we have $\omega_X|_F\cong\omega_F\cong\OO_F(-2)$. 
Tensoring the following exact sequence with $\omega_X|_F$:
\[ 0\ra \II/\II^2\ra \Omega_X^1|_F\ra \omega_F\ra 0, \]
we have the exact sequence:
\[ 0\ra \OO_F(-2)\ra \Omega_X^1\otimes\omega_X|_F\ra \OO_F(-4)\ra 0. \]
Taking the long exact sequence of cohomology groups, we have 
$H^0(F,\Omega_X^1\otimes\omega_X|_F)=0$, which is a contradiction.
\end{proof}

The main technical result in this paper is Proposition \ref{2.6}, 
which shows that certain log pairs on smooth rational surfaces are 
liftable to $W_2(k)$. The proof of Proposition \ref{2.6} is proceeded 
by induction via Lemma \ref{2.4}, however the initial step, 
i.e.\ Lemma \ref{2.3}, is proved by a argument, 
which depends on the geometric properties of Hirzebruch surfaces. 
Therefore it seems impossible to generalize Proposition \ref{2.6} 
to general surfaces. In fact, Proposition \ref{2.6} fails even for 
certain ruled surfaces, which is described in Corollary \ref{1.10}, 
since there exist counterexamples to the Kawamata-Viehweg vanishing 
on those ruled surfaces.

\begin{proof}[Proof of Corollary \ref{1.10}]
We use the same notation and construction as in \cite[Theorem 3.1]{xie07}. 
Therefore, there are a $\PP^1$-bundle $f:X\ra C$ and an ample $\Q$-divisor 
$H$ on $X$ with $\Supp(\langle H\rangle)=C'$ and $H^1(X,K_X+\ulcorner H
\urcorner)\neq 0$, where $C'\subset X$ is a smooth curve and 
$f|_{C'}:C'\ra C$ is the $k$-linear Frobenius morphism. 
By Theorem \ref{1.3}, $(X,C')$ cannot be lifted to $W_2(k)$.
\end{proof}

Note that Corollary \ref{1.10} means 
$0\neq o(X,C')\in H^2(X,\TT_X(-\log C'))$, 
while the $\PP^1$-bundle $X$ itself is liftable 
to $W_2(k)$ by Corollary \ref{4.2}.

Finally, we give some remarks on Theorem \ref{1.4}.

\begin{rem}\label{4.3}
(1) By a standard argument via Kodaira's lemma, Theorem \ref{1.4} gives 
rise to the Kawamata-Viehweg vanishing theorem for nef and big $\Q$-divisors 
on smooth rational surfaces, which may be useful in practice.
\vskip 2mm
($\dag$) Let $X$ be a smooth proper rational surface, and $L$ 
a nef and big $\Q$-divisor on $X$, such that the fractional part 
$\langle L\rangle$ has simple normal crossing support. 
Then $H^1(X,K_X+\ulcorner L\urcorner)=0$ holds.
\vskip 2mm

(2) The following Kodaira-Ramanujam vanishing theorem \cite{ra} 
is a special case of the Kawamata-Viehweg vanishing theorem for 
nef and big integral divisors on smooth surfaces.
\vskip 2mm
($\ddag$) Let $X$ be a smooth projective surface over an algebraically 
closed field $k$ with $\ch(k)=0$, and $L$ a nef and big integral divisor 
on $X$. Then $H^1(X,K_X+L)=0$ holds.
\vskip 2mm
By a result of Raynaud \cite[Corollaire 2.8]{di} and Corollary \ref{4.2}, 
the Kodaira-Ramanujam vanishing theorem holds on all smooth projective 
surfaces with negative Kodaira dimension in positive characteristic, 
while among those surfaces, there exist counterexamples to the 
Kawamata-Viehweg vanishing theorem for nef and big $\Q$-divisors 
(see \cite[Theorem 3.1]{xie07}). This observation shows that there 
is a significant difference between the $\Q$-divisor version and the 
integral divisor version of the Kawamata-Viehweg vanishing theorem 
in positive characteristic.
\end{rem}

\textsc{School of Mathematical Sciences, Fudan University, 
Shanghai 200433, China}

\textit{E-mail address}: \texttt{xie\_qihong@hotmail.com, qhxie@fudan.edu.cn}

\end{document}